\documentclass[11pt,reqno]{amsart}
\newtheorem{theorem}{Theorem}[section]
\newtheorem{proposition}[theorem]{Proposition}
\newtheorem{corollary}[theorem]{Corollary}
\newtheorem{remark}[theorem]{Remark}
\newtheorem{lemma}[theorem]{Lemma}
\newtheorem{example}[theorem]{Example}
\newtheorem{definition}[theorem]{Definition}

\usepackage{amssymb}
\usepackage{graphicx}
\usepackage{amsmath}
\usepackage[all]{xy}
\usepackage{enumerate}
\usepackage{amscd}
\usepackage{cases}

\usepackage[all]{xy}
\usepackage[latin1]{inputenc} 
\usepackage{graphicx} 
\usepackage{mathtools}
\numberwithin{equation}{section}

\begin{document}

\title[Lipeng Luo\textsuperscript{1}, Zhixiang Wu\textsuperscript{2}]{  Hopf action on vertex algebras}
\author{Lipeng Luo\textsuperscript{1} and Zhixiang Wu\textsuperscript{2}}

\address{\textsuperscript{1}Department of Mathematics, Zhejiang University, Hangzhou, Zhejiang Province,310027,PR China.}
\address{\textsuperscript{2}Department of Mathematics, Zhejiang University, Hangzhou, Zhejiang Province,310027,PR China.}

\email{\textsuperscript{1}luolipeng@zju.edu.cn}
\email{\textsuperscript{2}wzx@zju.edu.cn}



\date{\today}
\thanks{ This work was supported by the National Natural Science Foundation of China (No. 11871421) and the Zhejiang Provincial Natural Science Foundation of China (No. LQ16A010011).}

\begin{abstract}
In present paper,   some properties of Hopf action  on vertex algebras will be given. We prove that $V\#H$ is an $\mathcal{S}$-local vertex algebra but it is not a vertex algebra when $V$ is an $H$-module vertex algebra.   In addition, we extend the $\mathcal{S}$-locality to the quantum vertex algebras.
\end{abstract}

\footnote{The second author is the corresponding author.}
\maketitle

\section{Introduction}\label{intro}

Vertex algebras, which were introduced by Borcherds in \cite{Bor1}, have close connections to Lie conformal algebras, infinite-dimensional Lie algebras satisfying the locality property in \cite{KacL}.  Vertex (operator) algebras are analogous to classical Lie algebras and associative algebras in many aspects, though there are also significant differences in many aspects. For a vertex (operator) algebra $V$, one  can study the action of a vertex (operator) algebra by a Hopf algebra $H$. On this aspect,  Manson and Dong have established the Galois theory of  vertex operator algebras, which is similar to  the classical Galois theory of fields action by groups (see \cite{DM}). Like to the Galois theory of fields related to fields and groups was generalized to the Galois theory related associative algebras and Hopf algebras, the previous Galois theory about vertex operator algebras and groups can be generalized to a theory about vertex operator algebras and Hopf algebras. At first step, one need to introduce the concept of $H$-module vertex operator algebras. This was completed in \cite{DW, HW} by Dong and Wang. They also introduced $H$-module vertex algebras for a Hopf algebra $H$ and a smash product  $V\#H$ for  an  $H$-module vertex operator algebra $V$.  The next step is to study the relations between $V^H$ (the fixed vertex operator subalgebra under the action of $H$), $V$ and $V\#H$.  However, the smash is not a vertex operator algebra but an $\mathcal{S}$-local vertex operator algebra.  The $\mathcal{S}$-locality is defined as follow. 
For any $a,b \in V$, there exists $n,m \in  \mathbb{Z^+}$, and $a^i,b^i \in V, i=1,2,...,m$ such that 
\begin{eqnarray}\label{eq1.1}
(z-w)^nY(a,z)Y(b,w)=(z-w)^n\sum_{i=1}^mY(b^i,w)Y(a^i,z)
\end{eqnarray}
holds. An $\mathcal{S}$-local vertex algebra satisfies all the axioms for vertex algebras except that the locality is replaced by identity (\ref{eq1.1}).
The $\mathcal{S}$-local vertex algebra is similar to $\mathcal{S}$-local vertex operator algebra introduced in \cite{DW, HW}. 
The final step is to introduce the $H$-extensions of vertex operator algebras. In present paper, we continue to study $V^H$ and $V\# H$ for the preparation of concept of $H$-extensions of vertex operator algebras.
We prove that an $\mathcal{S}$-local vertex algebra $V$ can induce a unital associative algebra $A(V)$ with respect to the $f$-product, which is introduced in \cite{DGK, EH}. The smash product $V\#H$ of a Hopf algebra $H$ and an $H$-module vertex algebra $V$ was constructed. We prove that $V\#H$ is an $\mathcal{S}$-local vertex algebra. Moreover, $A(V\#H)\cong A(V)\#H$, the smash product of $H$ with its module algebra $A(V)$. 

The rest of the paper is organized as follows. In Section 2, we introduce some basic definitions, notations, and related known results about vertex algebras from \cite{BV, DGK, HL3}, and propose the definition of $\mathcal{S}$-local vertex algebras and all kinds of modules over them. More interesting is that we prove that an $\mathcal{S}$-local vertex algebra $V$ can induce a unital associative algebra $A(V)$ with respect to the $f$-product, which is introduced in \cite{EH}. In Section 3, Hopf action on field algebras and some basic
properties of this action are characterized. For a Hopf algebra $H$ and an $H$-module field algebra $V$, the construction of the smash product $V\#H$  are characterized. Furthermore, we prove that $V\#H$ is still a field algebra. In Section 4, the smash product $V\#H$ of a Hopf algebra $H$ and an $H$-module vertex algebra $V$ was constructed. We prove that $V\#H$ is an $\mathcal{S}$-local vertex algebra. Moreover, the connection between a $V$-module $M$, which is also an $H$-module and a $V\#H$-module was characterized when $V$ is an $H$-module vertex algebra. In addition, we prove that $A(V\#H)\cong A(V)\#H$, which is the smash product of $H$ with its module algebra $A(V)$. In Section 5, we extend the $\mathcal{S}$-locality to the quantum vertex algebras. The construction of $\mathcal{S}$-local quantum vertex algebra were determined.

Throughout this paper, we use notations ${\bf k}$, $\mathbb{Z}$ and $\mathbb{Z^{+}}$ to represent a field of characterization zero, the ring of integers and the set of nonnegative integers, respectively. In addition, all vector spaces and tensor products are over ${\bf k}$. In the absence of ambiguity, we abbreviate $\otimes_{{\bf k}}$ to $\otimes$.

\section{Basic on filed algebras}

Vertex algebras are field algebras with some local conditions.
 In this section, we recall some basic definitions, notations and related results about field algebras and their representations for later use. 
 
 Let ${\bf k}$ be a field of characterization zero and $R={\bf k}[s]$ the polynomial ring with variable $s$. Then $R$ is a Hopf algebra with  the coproduct determined by with $\Delta(s)=s\otimes 1+1\otimes s$ and the  counit determined by $\varepsilon(s)=0$.  
For any vector space $W$ over ${\bf k}$, let $W[[z,z^{-1}]]:=\{\sum\limits_{n\in \mathbb{Z}}w_nz^{-n-1}|w_n\in W\}$, $W((z)):=\{\sum\limits_{i=-n}^{+\infty}w_iz^i| w_i\in W$ for some positive integer $n\}$ and $W[[z]]:=\{\sum\limits_{n=0}^{+\infty}w_nz^n|w_n\in W\}$. The formal delta function is defined to be
$$\delta(z,w)=\sum\limits_{n\in\mathbb{Z}}z^{-n-1}w^n\in{\bf k}[[z^{\pm1},w^{\pm1}]].$$
The ring morphism $$i_{z,w}:{\bf k}[[z,w]][z^{-1},w^{-1},(z-w)^{-1}]\to{\bf k}((z))((w))$$is defined by Taylor expansion of $(z-w)^n$ in positive power of $w$.

Now, we recall the definition of a field algebra and its properties following \cite{BV, DGK}.

\begin{definition}\label{def2.31}(Ref. \cite{BV, DGK})
\begin{em}
Suppose $V$ is a left $R$-module and ${\bf 1}\in V$ a fixed element. A \emph{state-field correspondence} on a pointed vector space $V$ is a linear map $Y: V\otimes V \to V((z))$, $a \otimes b \mapsto Y(a,z)b=\sum_na_nbz^{-n-1}$ and the action of $s$ on $V$, satisfying the following axioms:
\begin{enumerate}[(i)]
\item (Vacuum axioms) $Y({\bf 1},z)a=a$, and $Y(a,z){\bf 1} =a+(sa)z+\cdots\in V[[z]]$ for any $a \in V$.
\item (Translation convariance) $sY(a,z)b-Y(a,z)sb=Y(sa,z)b=\partial_zY(a,z)b$, where $sa:=\partial_zY(a,z){\bf 1} |_{z=0}$.
\end{enumerate}
The action of $s$ on $V$ induces a linear endomorphism of $V$, usually denoted by  $T$, is called the \emph{translation operator}.

\end{em}		
\end{definition}

Let $V$ be a unital algebra with a unit element ${\bf 1}$, and let $s$ be a derivation of $V$. Define 
\begin{align}
Y(a,z)b=(e^{zs}a)b, \quad a,b \in V.
\end{align}
It is easy to check that $(V,{\bf 1}, s)$ is a state-field correspondence. Actually, all state-field correspondences for which the fields $Y(a,z)$ are formal power series in $z$ are obtained in this way. We call such a state-field correspondence \emph{trivial}.

\begin{definition}\label{def2.32}(Ref. \cite{BV, DGK})
\begin{em}
Let $(V,{\bf 1})$ be a pointed vector space with a state-field correspondence $Y$. Then $V$ is called a \emph {field algebra} (which is equivalent to the notion of \emph {nonlocal vertex algebra} in \cite{HL3}) if $Y$ satisfies the \emph{associativity relation}:
for any $a,b,c\in V$, there exists $n \in  \mathbb{Z^+}$ such that 
\begin{eqnarray}\label{eq2.32}
(z-w)^ni_{z,w}Y(a,z-w)Y(b,-w)c=(z-w)^nY(Y(a,z)b,-w)c
\end{eqnarray}
holds. 

A strong field algebra is a field algebra satisfying $n$-th product axiom
$$Y(a_{n}b,z)=Y(a,z)_{n}Y(b,z)$$
for any $n\in\mathbb{Z}$, where 
$$Y(a,z)_{m}Y(b,z)=Res_x [Y(a,x),Y(b,z)](x-z)^m,$$
$$Y(a,z)_{-m-1}Y(b,z)=:\partial^m_zY(a,z)Y(b,z):/m!,$$
for any $m\in\mathbb{Z}^+$.

Here $:\  :$ stands for the normal ordered product (or  Wick product) given by 
$$:Y(a,z)Y(b,z):=Y(a,z)_+Y(b,z)+Y(b,z)Y(a,z)_-,$$
where $$Y(a,z)_+=\sum\limits_{j\leq -1}a_{(j)}z^{-j-1}, \qquad  Y(a,z)_-=\sum\limits_{j\geq 0}a_{(j)}z^{-j-1}.$$

\end{em}		
\end{definition}

Let $Y$ be a state-field correspondence on a pointed vector space $(V,{\bf 1})$. Define 
\begin{align}\label{eq2.33}
Y^{op}(a,z)b=e^{zs}Y(b,-z)a.
\end{align}
It is not difficult to check that $Y^{op}$ is also a state-field correspondence. About the state-field correspondence, we also have the following properties followed in \cite{DGK}.
\begin{proposition}\label{def2.34}(Ref. \cite{BV, DGK})
Let $Y$ be a state-field correspondence on a pointed vector space $(V,{\bf 1})$. Then $Y$ satisfies the associativity relation (\ref{eq2.32}) if and only if all pairs $(Y(a,z)$,$Y^{op}(b,z))$ are local on each $c \in V$, i.e., there exists $n \in  \mathbb{Z^+}$ such that
\begin{align}\label{eq2.34}
(z-w)^n[Y(a,z),Y^{op}(b,w)]c=0.
\end{align}
		
\end{proposition}

 Similar to the opposite state-field correspondence, if $(V, {\bf 1}, Y)$ is a field algebra, then $(V, {\bf 1}, Y^{op})$ is also a field algebra. Now, we give the definition of vertex algebra.
\begin{definition}\label{def2.3}(Ref. \cite{BV, DGK})
\begin{em}
A vertex algebra is a pointed vector space $(V,{\bf 1})$ with a state-field correspondence $Y: V\otimes V \to V((z))$, such that each pair of $(Y(a,z)$,$Y(b,z))(a,b \in V)$ is local, i.e., there exists $n \in  \mathbb{Z^+}$ such that
\begin{align}\label{eq2.3}
(z-w)^n[Y(a,z),Y(b,w)]=0.
\end{align}
\end{em}		
\end{definition}

It was proved that a vertex algebra $(V, {\bf 1}, Y)$ is a strong field algebra with $Y(a,z)b=e^{zs}Y(b,-z)a$ for any $a,b\in V$. Moreover, a field algebra is a vertex algebra if and only if $Y(a,z)b=e^{zs}Y(b,-z)a$ for any $a,b\in V$. 

\begin{example}
Suppose that $(V, {\bf 1}, Y)$ is a field algebra and $H$ is a Hopf algebra. Then $V\otimes H$ is a field algebra with ${\bf 1} \otimes 1$ as vacuum vector, $s(a\otimes h)=sa\otimes h$ and $$Y^H(a\otimes f,z)(b\otimes h)= Y(a,z)b\otimes fh=\sum\limits_{n\in\mathbb{Z}}a_{n}b\otimes fhz^{-n-1}.$$  If $(V,Y,{\bf 1})$ is a strong field algebra, then $(V\otimes H,Y^H,{\bf 1} \otimes 1)$ is a strong field algebra.
Furthermore, if $(V,Y,{\bf 1})$ is a vertex algebra and $H$ is a commutative Hopf algebra, then $(V\otimes H,Y^H,{\bf 1} \otimes 1)$ 
is a vertex algebra.
\end{example}

Actually, a vertex algebra is simply a field algebra (nonlocal vertex algebra) that satisfies the locality identity (\ref{eq2.3}). By generalizing the local property, we can obtain the definition of weak quantum vertex algebra in \cite{HL3}.
\begin{definition}\label{def2.2}
\begin{em}
A field algebra $V$ is called a \emph {weak quantum vertex algebra} if it also satisfies the following weak $\mathcal{S}$-locality: 
for any $a,b \in V$, there exists $m \in  \mathbb{Z^+}$, and $f_i(x)\in {\bf k}((x)), a^i,b^i \in V, i=1,2,...,m$ such that
\begin{eqnarray}\label{eq2.2}
(z-w)^nY(a,z)Y(b,w)=(z-w)^n\sum_{i=1}^mf_i(z-w)Y(b^i,w)Y(a^i,z)
\end{eqnarray}
holds for some nonnegative integer $n$.

In particular, $V$ is called an \emph {$\mathcal{S}$-local vertex algebra} if it satisfies the following $\mathcal{S}$-locality identity instead of weak $\mathcal{S}$-locality (\ref{eq2.2}):
\begin{eqnarray}\label{eq2.4}
(z-w)^nY(a,z)Y(b,w)=(z-w)^n\sum_{i=1}^mY(b^i,w)Y(a^i,z).
\end{eqnarray}

\end{em}		
\end{definition}

\begin{remark}\label{rm2.5}
\begin{em}
(1) It is revealed in \cite{HL1} that identity (\ref{eq2.32}) together with(\ref{eq2.4}) implies the following two identities.
\begin{align}
x^{-1}&\delta(\frac{z-w}{x})Y(a,z)Y(b,w)-x^{-1}\delta(\frac{-w+z}{x})\sum_{i=1}^mY(b^i,w)Y(a^i,z) \nonumber \\
&=z^{-1}\delta(\frac{w+x}{z})Y(Y(a,x)b,w),\label{eq2.5}\\ 
Y(a&,z)b=e^{zs}\sum_{i=1}^mY(b^i,-z)a^i. \label{eq2.6}
\end{align}
(2) $\mathcal{S}$-local vertex algebras are special examples of weak quantum vertex algebras defined in \cite{HL3}. 
\end{em}
\end{remark}


 Now, we consider an $\mathcal{S}$-local vertex algebra induced from $V$.

 \begin{example} \label{ex2.8}
 Suppose that $V$ is a vertex algebra. For any $n \in \mathbb{Z^+}$, set $M(n, V)=V \otimes M(n, {\bf k})$, where $M(n, {\bf k})$ is the full matrix algebra of order $n$ over ${\bf k}$. Then $M(n, V)$ is a field algebra with
 \begin{align}\label{eq5.3}
 Y^{M(n, V)}(A,z)B=(Y(v_{ij},z))B,
 \end{align}
 where $A=(v_{ij})_{i,j=1}^n, B=(u_{ij})_{i,j=1}^n$, and  $v_{ij}, u_{ij} \in V$.
 In particular, $(M(n, V), Y^{M(n, V)}, {\bf 1}  E_n, sE_n)$ is an $\mathcal{S}$-local vertex algebra. We use $E_n$ to denote the identity matrix and $u(a_{ij})_{i,j=1}^n=(a_{ij}u)_{i,j=1}^n$ for $u\in V, a_{ij} \in {\bf k}$.
 \end{example}

In \cite{DW, HW}, Dong and Wang investigated an associative algebra $A_n(V)$, which was induced from an ($\mathcal{S}$-local) vertex operator algebra $V$, under the multiplication $*_n$ with unite $[{\bf 1}]$. 
\begin{proposition}(Ref. \cite{DW, HW})\label{pro2.9}
Let $(V, Y, {\bf 1}, \omega)$ be an ($\mathcal{S}$-local) vertex operator algebra. For any $a,b \in V$ homogeneous, and $n \in \mathbb{Z}^+$, define \begin{align*}
&a{*_n}b=\sum_{m=0}^n(-1)^m\tbinom{m+n}{m} Res_zY(a,z)b\frac{(1+z)^{wtu+n}}{z^{n+m+1}},\\
&a{\circ_n}b=Res_zY(a,z)b\frac{(1+z)^{wtu+n}}{z^{2n+2}},
\end{align*}
extend $*_n, \circ_n$ to $V$ by bilinearity. Set $O_n(V)=Span\{ a{\circ_n}b, L(-1)a+L(0)a |  \forall a,b \in V\}$. Then the quotient $A_n(V)=V/O_n(V)$ is an associative algebra with respect to multiplication $*_n$ and $[{\bf 1}]$ is a unit.
\end{proposition}

As mentioned above, it is not difficult to think that we can do the similar thing in an ($\mathcal{S}$-local) vertex algebra. However, we are committed to get a new algebra structure by changing the multiplication $*_n$ in $V$. Ekeren and Heluani determined another associative quotients of vertex algebras in \cite{EH} as follows:
\begin{proposition}(Ref. \cite{EH})\label{pro2.10}
Let $(V, Y, {\bf 1}, s)$ be a vertex algebra. For any $a,b \in V$, and $f(z)=z^{-1}$, or $f(z)=c\frac{e^{cz}}{e^{cz}-1}$ for some $c\neq 0$, and $g=\partial f$, we define the $f$-product
\begin{align}\label{eq2.7}
a_{(f)}b=Res_zf(z)Y(a,z)b.
\end{align}
Set $V_{(g)}V=\langle a_{(g)}b | \forall a,b \in V\rangle$. Then the quotient $A(V)=V/V_{(g)}V$ is an associative algebra with respect to the $f$-product and $[{\bf 1}]$ is a unit.
\end{proposition}
As we described in Definition\ref{def2.2}, we can obtain a new algebra by changing the locality identity. Inspired by \cite{DLM,EH,HW,YZ}, we can obtain the following theorem induced by Proposition \ref{pro2.9} and \ref{pro2.10}.

\begin{theorem}\label{tm2.6}
Let $(V, Y, {\bf 1}, s)$ be an $\mathcal{S}$-local vertex algebra. Then the quotient $A(V)=V/V_{(g)}V$ is an associative algebra with respect to the $f$-product and $[{\bf 1}]$ is a unit, where all the other conditions are the same as described in Proposition \ref{pro2.10}.
\end{theorem}
\begin{proof}
For any $a,b,c \in V$, set 
\begin{align*}
A(z,w)=Y(a,z)Y(b,w)c,\  B(z,w)=\sum_{i=1}^rY(b^i,w)Y(a^i,z)c\  \ and \  \ C(z,w)=Y(Y(a,z)b,w)c.
\end{align*}
According to the associativity relation (\ref{eq2.32}) and $\mathcal{S}$-locality identity (\ref{eq2.4}), we can obtain that 
\begin{align*}
(z-w)^NA(z,w)=(z-w)^NB(z,w)\ \  and \ \ (x+w)^Ni_{x,w}A(x+w,w)=(x+w)^NC(x,w).
\end{align*}
By Lemma 2.3 in \cite{DGK}, we can obtain that
\begin{align*}
&i_{z,w}\delta(x,z-w)Y(a,z)Y(b,w)c-i_{w,z}\delta(x,z-w)\sum_{i=1}^rY(b^i,w)Y(a^i,z)c\\
&=i_{z,x}\delta(w,z-x)Y(Y(a,x)b,w)c.
\end{align*}
Taking $Res_z$, we obtain that
\begin{align*}
&Res_z(i_{z,w}\delta(x,z-w)Y(a,z)Y(b,w)c-i_{w,z}\delta(x,z-w)\sum_{i=1}^rY(b^i,w)Y(a^i,z)c)=Y(Y(a,x)b,w)c,
\end{align*}
which is similar to the Borcherds identity.
Applying $Res_xx^n$ to both sides,we can obtain that 
\begin{align}\label{eq2.14}
&Res_z(i_{z,w}(z-w)^nY(a,z)Y(b,w)c-i_{w,z}(z-w)^n\sum_{i=1}^rY(b^i,w)Y(a^i,z)c)=Y(a_nb,w)c,
\end{align}
which is similar to the $n$-th product axiom. Similar to the proof of Proposition 4.1(a) in \cite{BV}, we can deduce that $Y$ satisfies the formula (\ref{eq2.14}) if and only if the following formula 
\begin{align}\label{eq2.15}
Y(a,z)Y^{op}(b,w)c-Y^{op}(b,w)Y(a,z)c=\sum_{i=1}^r\sum_{j\geq0}Y^{op}(a^i_jb,w)c^i\partial^{(j)}_w\delta(z-w)
\end{align}
 holds. According to the formula (\ref{eq2.33}) and (\ref{eq2.6}), we have 
\begin{align}\label{eq2.16}
Y^{op}(a,z)b=e^{zs}Y(b,-z)a=e^{zs}e^{-zs}\sum_{i=1}^rY(a^i,z)b^i=\sum_{i=1}^rY(a^i,z)b^i.
\end{align}

Using (\ref{eq2.14}), (\ref{eq2.15}) and (\ref{eq2.16}), we can deduce that the only universal way to guarantee $(a_{(g)}b)_{(f)}c, a_{(f)}(b_{(g)}c)  \in V_{(g)}V$ is to require 
$$f(z)\partial_z^{(j)}g(-z), \ g(z)\partial_z^{(j)}f(z) \in \langle \partial_z^{(k)}g(z) | k \in \mathbb{Z^+}\rangle$$
for all $j\in \mathbb{Z^+}$. 

The rest proof is similar to the proof of Proposition \ref{pro2.10}. 
\end{proof}

Now we give the module over the above algebras.


\begin{definition}
\begin{em}
A left module over a field algebra $(V,Y,{\bf 1})$ is a left $R$-module $M$ equipped with an operator $s \in EndM$ and a linear map

$$V\to EndM[[z,z^{-1}]], a\mapsto Y^M(a,z)=\sum\limits_{n\in\mathbb{Z}}a^M_{n}z^{-n-1},$$satisfying  
\begin{enumerate}[(i)]
\item For any $a\in V, v \in M$, $a_nv=0$ if $n \gg 0$.
\item (Vacuum axiom) $Y^M({\bf 1},z)=id_M$.
\item (Translation invariance) $sY^M(a,z)-Y^M(a,z)s=\partial_zY^M(a,z)$.
\item (Associativity axiom) $(z-w)^NY^M(Y(a,z)b,-w)v=(z-w)^Ni_{z,w}Y^M(a,z-w)Y^M(b,-w)v$ for $a,b\in V,v\in M$, $N\gg 0$.
\end{enumerate}
A right module $M$ over $(V,Y,{\bf 1})$, we replace the associative axiom of left module  with the following

($iv^{'}$) Associativity axiom, $$(z-w)^NY^M(e^{zs}Y(b,-z)a,-w)v=(z-w)^Ni_{z,w}Y^M(a,z-w)Y^M(b,-w)v$$ for $a,b\in V,v\in M$, $N\gg 0$.

When $(V,Y,{\bf 1})$ is a strong field algebra, in the above definition we replace the associative axiom with the following $n$-th axiom:
$$Y^M(a_{n}b,z)=Y^M(a,z)_{n}Y^M(b,z), \qquad n\in\mathbb{Z}.$$
\end{em}		
\end{definition}

Two modules $M_1$ and $M_2$ for $(V,Y,{\bf 1})$ are isomorphic if there is an invertible $R$-module map $f:M_1\to M_2$ such that $fY^{M_1}(a,z)=Y^{M_2}(a,z)f$ for $a\in V$.
The notation $M_1\simeq M_2$ indicates that $M_1$ and $M_2$ are isomorphic. There is the obvious notion of simple $(V,Y,{\bf 1})$-module, and in particular $(V,Y,{\bf 1})$ is called simple if it is itself a simple left $(V,Y,{\bf 1})$-module and simple right $(V,Y,{\bf 1})$-module.

For any subset $S\subseteq M$ of a representation of  the field algebra $(V,Y,{\bf 1})$, let $\langle S\rangle$ be the $(V,Y,{\bf 1})$-submodule of $M$ generated by $S$, i.e., the smallest $(V,Y,{\bf 1})$-submodule of $M$ contains $S$. Obviously, $\langle S \rangle $ is spanned by all elements of the form
$$v^1_{n_1}\cdots v^t_{n_t}a,\qquad v\in V, n_j\in \mathbb{Z}, a\in S, t\ge 0.$$
Similar to Proposition 4.1 of \cite{DM},  we can prove that $\langle S\rangle $ is spanned by 
$$v_{n}a,\qquad v\in V, n\in \mathbb{Z}, a\in S.$$

\begin{example}Suppose $(M,Y^M)$ is a representation of a field algebra $(V,Y,{\bf 1})$, and $N$ is a vector space. Then $M\otimes N$ is a representation of 
$(V,Y,{\bf 1})$ with $s(m\otimes x)=(sm)\otimes x$ and the linear mapping 
$Y^{M\otimes N}(a,z)=\sum\limits_{n\in\mathbb{Z}}a^{M\otimes N}_{n}z^{-n-1},$  where $a^{M\otimes N}_{n}(m'\otimes y)=a^M_{n}(m')\otimes y$.
\end{example}

Let  $g$ be left $R$-module  automorphism of  $(V,Y,{\bf 1})$. Then $g$ is called an automorphism of field algebra $(V,Y,{\bf 1})$ if $g({\bf 1})={\bf 1}$  
and $Y(g(a),z)g(b)=\sum\limits_{n\in\mathbb{Z}}g(a_{n}b)z^{-n-1}$, or equivalently, $g(a)_{n}g(b)=g(a_{n}b)$ for all $n\in\mathbb{Z}$.
If  $f\in End(V)$ is an endomorphism of the left $R$-module $V$, then $f$ is called a derivation of the field algebra $(V,Y,{\bf 1})$ if $Y(f(a),z)b+Y(a,z)f(b)=0$, $f({\bf 1})=0$.
That $Y(f(a),z)b+Y(a,z)f(b)=0$ means that $f(a_{n}b)=f(a)_{n}b+a_{n}f(b)$ for all $n\in\mathbb{Z}$.

A state-field correspondence $Y$ on a left $R$-module $(V, {\bf 1})$ is holomorphic if $Y(a,z)b\in V[[z]]$ for all $a,b\in V$. Suppose $(V,{\bf 1})$ is a torsion $R$-module. Then $Y$ is holomorphic.

\begin{proposition}Suppose $(M,Y^M)$ is a representation of a field algebra $(V, Y,{\bf 1})$. If $M$ is a finite-dimensional, then $Y^M(a,z)m\in M[[z]]$ for any $a\in V$ and $m\in M$.
\end{proposition}
\begin{proof} 
For any $a\in V$ and $m\in M$, let $[a,m]=\sum\limits_{n=0}^{+\infty}(-s)^{(n)}\otimes 1\otimes_Ha_{(n)}m$, where $(-s)^{(n)}=\frac{(-s)^n}{n!}$. Then
$[sa,m]=(s\otimes 1\otimes _H1)[a,m]$ and $[a,sm]=(1\otimes s\otimes_H1)[a,m]$. As $M$ is finite-dimensional, we have $[a,m]=0$ for any $a\in V, m\in  M$.
\end{proof}

Now, we need to introduce the following two definitions.
\begin{definition}\label{def2.8}(Ref. \cite{BV, DGK})
\begin{em}
A left module over a vertex algebra $(V, Y, {\bf 1}, s)$ is actually a left $V$-module $(M, Y^M)$ as field algebra, such that each pair of $(Y^M(a,z)$,$Y^M(b,z))(a,b \in V)$ is local, i.e., there exists $n \in  \mathbb{Z^+}$ such that
\begin{align}\label{eq2.333}
(z-w)^n[Y^M(a,z),Y^M(b,w)]=0.
\end{align}
\end{em}		
\end{definition}

\begin{definition}\label{def2.9}
\begin{em}
Let $V$ be an $\mathcal{S}$-local vertex algebra and let $(M, Y^M)$ be a left $V$-module as field algebra. Then $M$ is called a left module for $V$ as $\mathcal{S}$-local vertex algebra, if it satisfies the following property: for any $a,b \in V$, there exists $n,m \in  \mathbb{Z^+}$, and $a^i,b^i \in V, i=1,2,...,m$ such that the following $\mathcal{S}$-locality indentity
\begin{eqnarray}\label{eq2.10}
(z-w)^nY^M(a,z)Y^M(b,w)=(z-w)^n\sum_{i=1}^mY^M(b^i,w)Y^M(a^i,z).
\end{eqnarray}

\end{em}		
\end{definition}

\section{Hopf actions on field algebras}

In this section, we define the action and coaction of a Hopf algebra on a field algebra. Let $H$  be a Hopf algebra with coproduct $\Delta$, counit $\varepsilon$ and antipode $S$. For any $h\in H$, $\Delta(h)= \sum h_{1}\otimes h_{2}.$

\begin{definition}\label{def3.1}
Let $(V,Y,{\bf 1})$ be a field algebra and $H$ a Hopf algebra. We say that $V$ is an $H$-module field algebra, if 

(i) V is a left $H$-module.

(ii) $h{\bf 1}=\varepsilon(h){\bf 1}$.

(iii) $hsa=sha$.

(iv) $h(a_{n}b)=\sum (h_{1}a)_{n}(h_{2}b)$ for any $h\in H$ and $a,b\in V$.

\noindent $V$ is called an $H$-comodule field algebra if 

(i) $V$ is a right $H$-comodule with $\rho:V\to V\otimes H, v\mapsto \sum v_{0}\otimes v_{1}$.

(ii) $\rho({\bf 1})={\bf 1}\otimes 1$.

(iii) $\rho(a_{n}b)=\sum {(a_0)}_{n}(b_{0})\otimes a_{1}b_{1}$.

(iv) $\rho(sa)=\sum sa_{0}\otimes a_{1}$ for any $a,b\in V$ and any $n\in \mathbb{Z}$, where $\rho(a)=\sum a_{0}\otimes a_{1}\in V\otimes H$.

\noindent In particular, $V$ is called an $H$-module  (resp. $H$-comodule) ($\mathcal{S}$-local) vertex algebra if $V$ is a ($\mathcal{S}$-local) vertex algebra.
\end{definition}

\begin{example}Let $(V,Y, {\bf 1})$ be a field algebra. Then $(V\otimes H,Y^H,{\bf 1} \otimes 1)$ is an $H$-module field algebra with $Y^H(a\otimes g,z)(b\otimes h)=Y(a,z)b\otimes gh$ and  $h(a\otimes f):=\sum a\otimes h_{1}fS(h_{2})$. 

Suppose $(V,Y,{\bf 1})$ is an $H$-module field algebra and $H$ is cocommutative. Then $(V, Y^{op},{\bf 1})$ is also an $H$-module field algebra, where $Y^{op}(a,z)b:=e^{zs}Y(b,-z)a$.

Suppose that $G$ is a finite group. The $G$-graded VOA introduced in \cite{DM} is a ${\bf k}[G]$-comodule field algebra with the coaction given by $\rho(v)=v\otimes g$ for any $v\in V_{(g,n)}$.
\end{example}

Now, we give the following results for later use.

\begin{lemma}\label{lm2.16}(Ref. \cite{HL2})
Let $H$ be a Hopf algebra and $V$ an $H$-module field algebra. Set $V\#H=V\otimes H$ as vector space. If $(M, Y_M, s_M)$ is a $V$-module as well as an $H$-module such that $h(Y_M(a,z)m)=\sum Y_M(h_1a,z)h_2m$ and $hs_M(m)=s_M(hm)$ for any $h\in H$, $a\in V$, $m \in M$, where $\Delta(h)=\sum h_1\otimes h_2$. Then $(M, Y_M^{V\#H}, s_M)$ is a $V\#H$-module, where $Y_M^{V\#H}$ is defined as follows:
\begin{align*}
Y_M^{V\#H}(a\#h,z)m=Y_M(a,z)hm
\end{align*}
 for any $h\in H$, $a\in V$, $m \in M$.
\end{lemma}

 Suppose that $V$ is an $H$-module vertex algebra and $H={\bf k}G$, where $G$ is a finite subgroup of $Aut(V)$. Then $H^*$ is also a Hopf algebra with 
 \begin{align*}
 \rho_g\rho_h=\delta_{g,h}\rho_g,\quad \Delta(\rho_g)=\sum_{h \in G}\rho_{gh^{-1}}\otimes \rho_h,\quad \varepsilon(\rho_g)=\delta_{1,g},\quad S(\rho_g)=\rho_{g^{-1}},
 \end{align*}
 where $\rho_g \in H^*$ is defined by $\rho_g(h)=\delta_{g,h}$, for any $g,h \in G$. It is easy to check that $\sum_{g \in G}\rho_g$ is the identity element of $H^*$ which will be denoted by $E$. Thus, we can obtain following results.
 \begin{lemma}\label{lm5.1}
Let $V$ be an $H$-module vertex algebra and $H={\bf k}G$, where $G$ is a finite subgroup of $Aut(V)$. Then 
\begin{enumerate}[(i)]
\item $V\#H$ is an $H^*$-module, where the action of $H^*$ on $V\#H$ is defined as $\rho_h(a\#g)=\delta_{h,g}a\#g$ for any $a\#g \in V\#H, \rho_h \in H^*$ with $h,g \in G$;
\item  For any $f \in H^*, u\#g, v\#h \in V\# H$, $f(Y^{V\#H}(u\#g,z)v\#h)=\sum Y^{V\#H}(f_1(u\#g),z)(f_2(v\#h))$, where $\Delta(f)=\sum f_1 \otimes f_2$. 
 \end{enumerate}
 Furthermore, $V\#H$ is an $H^*$-module $\mathcal{S}$-local vertex algebra. 
 \end{lemma}
 \begin{proof}
 We just prove the second result, because the module structure is obvious. For convenience, we take $f=\rho_a$, then we can deduce that 
  \begin{align*}
 &\rho_a(Y^{V\#H}(u\#g,z)v\#h)=\rho_a(Y(u,z)(gv)\#gh)=\delta_{a,gh}Y^{V\#H}(u\#g,z)(v\#h),
  \end{align*}
 and
 \begin{align*}
 &\sum_{x \in G}Y^{V\#H}(\rho_{ax^{-1}}u\#g,z)\rho_xv\#h\\
 &= \sum_{x \in G}Y^{V\#H}(\delta_{ax^{-1},g}u\#g,z)(\delta_{x,h}v\#h)\\
 &=\sum_{x \in G}\delta_{ax^{-1},g}\delta_{x,h}Y^{V\#H}(u\#g,z)(v\#h)\\
 &=\delta_{ah^{-1},g}Y^{V\#H}(u\#g,z)(v\#h)\\
 &=\delta_{a,gh}Y^{V\#H}(u\#g,z)(v\#h).
 \end{align*}
  Thus, we have 
   \begin{align*}
   \rho_a(Y^{V\#H}(u\#g,z)v\#h)=\sum_{x \in G}Y^{V\#H}(\rho_{ax^{-1}}u\#g,z)\rho_xv\#h.
    \end{align*}
 This completes the proof.
 \end{proof}


Suppose that $H$ is a finite-dimensional Hopf algebra and $H^*$ is its dual Hopf algebra. Then $V$ is a comodule field algebra if and only if $V$ is an $H^*$-module field algebra.

\begin{lemma}\label{lm2.15}
Suppose that $(V, Y, {\bf 1})$ is an $H$-module field algebra. Then $(V^H,Y, {\bf 1})$ and $(V\otimes H,Y_H,{\bf 1} \otimes 1)$ are field algebras, where 
\begin{align}
Y_H(a\otimes h,z)(b\otimes g)=\sum\limits_{n\in\mathbb{Z}}a_{n}(h_{1}b)\otimes h_{2}gz^{-n-1},
\end{align}
for any $a\otimes h, b\otimes g \in V\otimes H$.
Moreover, if $t$ is a left integral of $H$ such that $\varepsilon(t)\neq 0$, then there is an integral $t$ such that $\varepsilon(t)=1$ and 
  $V^H\simeq \{({\bf 1} \otimes t)_{-1}(v_{-1}({\bf 1} \otimes t))|v\in V\otimes H\}$.  
  \end{lemma}

\begin{proof}The definition of $Y_H$ means that $Y_H(a\otimes h,z)=\sum\limits_{n\in\mathbb{Z}}a_{n}\circ h_{1}\otimes h_{2}\otimes z^{-n-1}$ and $(a\otimes h)_{n}(b\otimes g)=( a_{n}(h_{1}b))\otimes h_{2}g$. Thus, $Y_H({\bf 1}\otimes 1,z)=\sum\limits_{n\in\mathbb{Z}}{\bf 1}_{n}\otimes 1\otimes z^{-n-1}=id_{V\otimes H}$ and $Y_H(a\otimes h,z)({\bf 1}\otimes 1)=\sum\limits_{n\in\mathbb{Z}}a_{n}{\bf 1}\otimes h\otimes z^{-n-1}=a\otimes h+sa\otimes hz+\cdots\in V\otimes H[[z]].$
Define $s(a\otimes h)=sa\otimes h$. Then $Y_H(a\otimes h,z){\bf 1} \otimes 1=a\otimes h+s(a\otimes h)z+\cdots$   and 
\begin{align*}
& [s,Y(a\otimes h,z)](b\otimes g)\\
 &=\sum\limits_{n\in\mathbb{Z}}s(a_{n}h_{1}b)\otimes h_{2}g\otimes z^{-n-1}-\sum\limits_{n\in\mathbb{Z}}a_{n}h_{1}sb\otimes h_{2}g\otimes z^{-n-1}\\
 &=\sum\limits_{n\in\mathbb{Z}}(sa)_{n}h_{1}b\otimes h_{2}g\otimes z^{-n-1}\\
 &=Y(s(a\otimes h),z)(b\otimes g). 
 \end{align*}
 Hence $[s,Y(a\otimes h,z)]=Y(s(a\otimes h),z)$. Similarly, 
 \begin{align*}
 &\sum\limits_{n\in\mathbb{Z}}(sa)_{n}h_{1}b\otimes h_{2}g\otimes z^{-n-1}
 &=\sum\limits_{n\in\mathbb{Z}}-(n+1)a_{n}h_{1}b\otimes h_{2}g\otimes z^{-n-2}\
 &=\partial_z(Y(a\otimes h,z)(b\otimes g)). 
 \end{align*}
 Thus $[s,Y(a\otimes h,z)]=\partial_zY(a\otimes h,z)$.

Let $N$ be a positive integer such that $$(z-w)^NY(Y(a,z)h_{1}b,-w)(h_{2}g_{1}c)=(z-w)^Ni_{z,w}Y(a,z-w)Y(h_{1}b,-w)h_{2}g_{1}c.$$
Then 
\begin{align*}
&(z-w)^NY_H(Y_H(a\otimes h,z)(b\otimes g),-w)(c\otimes f)\\
&=(z-w)^N\sum\limits_{n,m\in\mathbb{Z}}(a_{n}h_{1}b)_{m}h_{2}g_{1}c\otimes h_{3}g_{2}f\otimes z^{-n-1}(-w)^{-m-1}\\
&=(z-w)^Ni_{z,w}Y(a,z-w)Y(h_{1}b,-w)h_{2}g_{1}c\otimes h_{3}g_{2}f\\
&=(z-w)^Ni_{z,w}Y_H(a\otimes h,z-w)Y_H(b\otimes g,-w)(c\otimes f).
\end{align*}
By now, we have proved that $(V\otimes H,Y_H, {\bf 1}\otimes 1)$ is a field algebra.

Suppose $t$ is a left integral of $H$ and $\varepsilon(t)\neq 0$. Then $t'=\varepsilon(t)^{-1}t$ satisfying $\varepsilon(t')=1$. We use $t$ to replace $t'$ for simplicity.
For any $v\otimes h\in V\otimes H$, we have 
\begin{align*}
({\bf 1} \otimes t)_{-1}((v\otimes h)_{-1}({\bf 1} \otimes t))=\varepsilon(h)({\bf 1}\otimes t)_{-1}(v\otimes t)=\varepsilon(h)tv\otimes t. 
\end{align*}
Let $\varphi:V^H\to\{({\bf 1}\otimes t)_{-1}(v_{-1}({\bf 1} \otimes t))|v\in V\otimes H\}, a\mapsto a\otimes t.$ 
Since $tv\in V^H$, $\varphi$ is surjective. If $a\in V^H$, then 
\begin{align*}
({\bf 1}\otimes t)_{-1}(a\otimes t_{-1}({\bf 1} \otimes t))=({\bf 1}\otimes t)_{-1}(a\otimes t)=a\otimes t.
\end{align*}
Thus $\varphi$ is bijective. Let $a,b\in V^H$. Then 
\begin{align*}
Y_H(a\otimes t,z)(b\otimes t)=\sum\limits_{n\in\mathbb{Z}}a_{n}t_{1}b\otimes t_{2}tz^{-n-1}=Y(a,z)b\otimes t.
\end{align*}
This means that $\varphi$ is an isomorphism of field algebras.
\end{proof}

\begin{remark}
It is necessary to point out that $(V^H,Y, {\bf 1})$ is a vertex algebra but $(V\otimes H,Y_H,{\bf 1} \otimes 1)$ is not a vertex algebra when $(V, Y, {\bf 1})$ is an $H$-module vertex algebra, because $Y_H\neq Y_H^{op}$. Fortunately, it is an $\mathcal{S}$-local vertex algebra see Theorem \ref{tm3.2} for details, when $(V, Y, {\bf 1})$ is an $H$-module vertex algebra.
\end{remark}

\begin{proposition}$(V,Y^V)$ is a representation of $(V\otimes H,Y_H,{\bf 1}\otimes 1)$ and $(V,Y)$ is a representation of $(V^H,Y,{\bf 1})$,
where $Y^V(a\otimes h,z)b:=Y(a,z)(hb)$.
\end{proposition}

\begin{proof} For any $a\otimes h,b\otimes g\in V\otimes H$ and $c\in V$, we have 
\begin{align*}
&[s, Y^V(a\otimes h,z)]c\\
&=sY^V(a\otimes h,z)c-Y^V(a\otimes h,z)sc\\
&=sY(a,z)hc-Y(a,z)shc\\
&=[s,Y(a,z)]hc\\
&=\partial_zY(a,z)hc\\
&=\partial_zY^V(a\otimes h,z)c.
\end{align*}
Hence $[s,Y^V(a\otimes h,z)]=\partial_zY^V(a\otimes h,z)$.
Moreover, we can obtain that 
\begin{align*}
&(z-w)^NY^V(Y_H(a\otimes h,z)(b\otimes g),-w)c\\
&=(z-w)^NY^V(Y(a,z)h_{1}b\otimes h_{2}g,-w)c\\
&=(z-w)^NY(Y(a,z)h_{-1}b,-w)h_{2}gc\\
&=(z-w)^Ni_{z,w}Y(a,z-w)Y(h_{1}b,-w)h_{2}gc\\
&=(z-w)^Ni_{z,w}Y(a,z-w)Y^V(h_{1}b\otimes h_{2}g,-w)c\\
&=(z-w)^Ni_{z,w}Y^V(a\otimes h,z-w)Y^V(b\otimes g,-w)c.
\end{align*}
Therefore, $(V,Y^V)$ is a representation of $(V\otimes H,Y_H,{\bf 1}\otimes 1)$.
\end{proof}

Notice that for any two left $H$-module $M_1,M_2$, $Hom_H(M_1,M_2)$ is also a left $H$-module via $(hf)(m)=\sum h_{1}(f(S(h_{2})m))$.

\begin{proposition} 
Let $\{S_i|i\in I\}$ be the set of all irreducible representations of $H$. Suppose the $H$-module field algebra $(V,Y,{\bf 1})$ is semisimple as $H$-module and ${\bf k}$ is algebraically closed. Then $V\simeq  \sum\limits_{i\in I}\oplus S_i\otimes Hom_H(S_i,V)$ and $V^H\simeq Hom_H({\bf k},V)$. Moreover, $H$ acts on $S_i$ and $V^H$ acts on $Hom_H(S_i,V)$.
\end{proposition}
\begin{proof}Suppose that $V=\sum\limits_{i\in I}\oplus V_i$ is a direct sum of homogenous $H$-modules, where $V_i$ is the sum of all isomorphic irreducible $H$-submodules which are isomorphic to $S_i$. Then $Hom_H(S_i,V)=Hom_H(S_i,V_i)$ and $V_i\simeq S_i\otimes Hom_H(S_i,V_i)$  for each $i\in I$. For any $h\in H$, $u\in V^H$ and $v\in V$, we have $hY(u,z)v=\sum Y(h_{1}u,z)h_{2}v=Y(u,z)hv$. Thus $ H$ acts on $S_i$ and $V^H$ acts on $Hom_H(S_i,V)$.\end{proof}

\begin{proposition} 
Let $V$ be an $H$-module field algebra. Then $V^H\otimes H$ is a field algebra, $V\otimes t$ is a left ideal of $V\otimes H$, $V^H\otimes t$ is a right ideal of $V^H\otimes H$, and $V^H\otimes t$ is an ideal of $V^H\otimes H$.
\end{proposition}

The field algebra $(V\otimes H,Y_H,{\bf 1}\otimes 1)$ is denoted by $V\# H$ in the sequel. As $(a\otimes h)_{n}(b\otimes t)=(a_{n}h_{1}b)\otimes h_{2}t=a_{n}hb\otimes t$. Thus $Y_H(a\otimes h,z)(b\otimes t)=\varepsilon(h)Y(a,z)b\otimes t$ for every $b\in V^H$. $({\bf 1}\otimes t)_{-1}(a\otimes t_{-1}({\bf 1} \otimes t))=\varepsilon(t)^{-1}({\bf 1}\otimes t)_{-1}(a\otimes t)=a\otimes t.$
$(b\otimes t)_{n}(a\otimes h)=b_{n}t_{1}a\otimes t_{2}h$. Thus $(b\otimes t)_{n}(a\otimes h)=b_{n}a\otimes th$ for any $a\in V^H$.

Since $V^H$ (resp. $V$) is a subalgebra of $V$ (resp. $V\#H$), $V$ (resp. $V\#H$) is a both  left and right  module over $(V^H, Y, {\bf 1})$ (resp. $(V,Y,{\bf 1})$).

Suppose a left $R$-module $M$ is a representation of an $H$-module field algebra $(V,Y,{\bf 1})$. Then $M$ is said to be a \emph{Hopf representation}  if   $$h(a_{n}m)=\sum(h_{1}a)_{n}(h_{2}m)$$ for any $h\in H$, $a\in V$, $m\in M$ and all $n\in \mathbb{Z}$.

\begin{proposition} Suppose $(M,Y^M)$ is a Hopf representation of a field algebra $(V,Y,{\bf 1})$. Then $(M\otimes H,Y^{M\otimes H})$ is a representation of   $(V\#H,Y_H, {\bf 1}\otimes 1)$ with 
$$Y^{M\otimes H}(a\otimes h,z)(b\otimes g)=Y^M(a,z)(h_{1}b)\otimes h_{2}g$$
for $a\otimes h\in V\#H,b\otimes g\in M\otimes H$.
\end{proposition}

\begin{proof}For any $a\otimes h,b\otimes g\in V\#H$ and $m\otimes f\in M\otimes H$, we have 
\begin{align*}
&(z-w)^NY^{M\otimes H}(Y(a\otimes h,z)(b\otimes g),-w)(m\otimes f)\\
&=(z-w)^NY^{M\otimes H}(\sum\limits_{n\in\mathbb{Z}}a_{n}(h_{1}b)\otimes h_{2}gz^{-n-1},-w)(m\otimes f)\\
&=(z-w)^NY^M(\sum\limits_{n\in\mathbb{Z}}a_{n}(h_{1}b)z^{-n-1},-w)(h_{2}g_{1}m)\otimes h_{3}g_{2}f\\
&=(z-w)^NY^M(Y(a,z)(h_{1}b),-w)(h_{2}g_{1}m)\otimes h_{3}g_{2}f\\
&=(z-w)^Ni_{z,w}Y^M(a,z-w)Y^M(h_{1}b,-w)(h_{2}g_{1}m)\otimes h_{2}g_{2}f\\
&=(z-w)^Ni_{z,w}Y^{M\otimes H}(a\otimes h,z-w)Y^{M\otimes H}(b\otimes g, -w)(m\otimes f).
\end{align*}
The other axioms are easy to check. Hence $(M\otimes H,Y^{M\otimes H})$ is a representation of $V\#H$.
\end{proof}

\begin{definition}Let $(V,Y,{\bf 1})$ be a right $H$-comdule field algebra. Then a total integral for $(V,Y,{\bf 1})$ is a right $H$-comodule map $\phi:H\to V$ such that 
$\phi(1)={\bf 1}$.
\end{definition}

\begin{proposition}Let $(V,Y,{\bf 1})$ be an $H$-module field algebra, where $H$ is finite-dimensional, and consider $(V,Y,{\bf 1})$ as a right $H^*$-comodule field 
algebra. Then $\hat{t}:V\to V^H$ is surjective if and only if there exists a total integral $H^*\to V$.
\end{proposition}

\begin{proof}Assume that $\hat{t}$ is surjective and $c\in V$ such that $tc={\bf 1}$. Suppose $\theta:H\to H^*$ is given by $h\mapsto (h\rightharpoonup T)$, where $T$ is the left integral of $H^*$. Then $\theta$ is a left $H$-module isomorphism. Define $\phi:H^*\to V$ via $f\mapsto \theta^{-1}(f)c$. Then $\phi$ is a left $H$-module mapping, and thus $\phi$ is a right $H^*$-comodule mapping. Moreover, $\phi(1_{H^*})=\phi(\varepsilon)=tc={\bf 1}$.

Conversely, if $\phi:H^*\to V$ is a total integral and $c=\phi(T)$. Then $tc=t\phi(T)=\phi(t\rightharpoonup T)=\phi(\varepsilon)={\bf 1}.$
\end{proof}

\begin{proposition}\label{prop37}Let $H$ be  finite-dimensional acting on $(V,Y,{\bf 1})$ and assume that $\hat{t}:V\to V^H$, $v\mapsto tv$ is surjective. Then there is a nonzero $e\in V\# H$ such that
\begin{enumerate}[(i)]
\item $Y_H(e,z)e=e$, $V^H\simeq (V^H\otimes 1)_{-1}e\simeq \{(e_{-1}v)_{-1}e|v\in V\#H\}$, 
\item $Y_H(Y_H(e,z)V\# H,w)e= Y_H(V^H((z))\otimes 1,w)e$,
\item $Y_H(Y_H(e,z)v\otimes 1,w)e=Y_H(e,z)Y_H(v\otimes 1,w)e=Y_H(v\otimes 1,w)e$ for $v\in V^H$
\end{enumerate}
holds.
\end{proposition}
\begin{proof}Since $\hat{t}$ is surjective, there is $c\in V$ such that $tc={\bf 1}$. Define $e:=\sum t_{1}c\otimes t_{2}$, where $\Delta(t)=\sum t_{1}\otimes t_{2}$.
Then $Y_H(e,z)e=\sum\limits_{n\in\mathbb{Z}}(t_{1}c)_{n}t_{2}t_{1}'c\otimes t_{3}t'_{2}z^{-n-1}=\sum\limits_{n\in\mathbb{Z}}(tc)_{n}t_{1}c\otimes t_{2}z^{-n-1}=e.$

For any $a\in V^H$, since $m(1\otimes \varepsilon)(a\otimes 1)_{-1}e=a_{-1}{\bf 1}=a$, $V^H\simeq (V^H\otimes 1)e$. Moreover, $Y_H((a\otimes 1)_{-1}e,z)((b\otimes1)_{-1}e)=(Y(a,z)b\otimes 1)_{-1}e$. This means that $V^H$ is isomorphic to $(V^H\otimes 1)_{-1}e$ as field algebras.

Define $\varphi:(V^H\otimes 1)_{-1}e\to \{(e_{-1}v)_{-1}e|v\in V\#H\}$ via $\varphi((a\otimes 1)_{-1}e):=(e_{-1}(a\otimes 1))_{-1}e$. Since $(e_{-1}(v\otimes h))_{-1}e=\varepsilon(h)(t(c_{-1}v)\otimes 1)_{-1}e$, $\varphi$ is surjective. In particular, if $a\in V^H$, then $\varphi((a\otimes 1)_{-1}e)=(tc_{-1}a\otimes 1)_{-1}e=(a\otimes 1)_{-1}e$. Hence $\varphi$ is an isomorphism of field algebras.

Let $v\otimes h\in V\# H$, we can obtain that
\begin{align*}
&Y_H(Y_H(e,z)v\otimes h,w)e=\sum\limits_{m,n\in\mathbb{Z}}(e_{n}(v\otimes h))_{m}ez^{-n-1}w^{-m-1}\\
&=\sum\limits_{m,n\in\mathbb{Z}}(t_{1}c_{n}t_{2}v)_{m}t_{3}h_{1}t_{1}'c\otimes t_{4}h_{2}t_{2}'z^{-n-1}w^{-m-1}\\
&=\varepsilon(h)\sum\limits_{m,n\in{\mathbb{Z}}}(t_{1}c_{n}t_{2}v)_{m}t'_{1}c\otimes t'_{2}z^{-n-1}w^{-m-1}\\
&=\varepsilon(h)Y_H(Y(t_{1}c,z)t_{2}v\otimes 1,w)e\in Y_H(V^H((z))\otimes 1,w)e,
\end{align*}
and
\begin{align*}
&Y_H(e,z)Y_H(v\otimes h,w)e=\sum\limits_{m\in\mathbb{Z}}Y_H(e,z)v_{m}h_{1}t_{1}c\otimes h_{2}t_{2}w^{-m-1}\\
&=\varepsilon(h)\sum\limits_{m,n\in\mathbb{Z}}t_{1}c_{n}(t_{2}(v_{m}t_{1}'c))\otimes t_{3}t_{2}'z^{-n-1}w^{-m-1}\\
&=\varepsilon(h)\sum\limits_{m,n\in\mathbb{Z}}t_{1}c_{n}(t_{2}v_{m}t_{(1)}'c)\otimes t_{2}'z^{-n-1}w^{-m-1}\\
&=\varepsilon(h)Y_H(t_{1}c\otimes 1,z)Y_H(t_{2}v\otimes 1,w)e.
\end{align*}

In particular, if $v\in V^H$, then $Y_H(Y_H(e,z)v\otimes 1,w)e=Y_H(e,z)Y_H(v\otimes 1,w)e=Y_H(v\otimes 1,w)e$. Thus $Y_H(Y_H(e,z)(V\# H),w)e=Y_H(V^H((z))\otimes 1,w)e$.
\end{proof}

Suppose $(V,Y,{\bf 1})$ is an $H$-module field algebra and $(t)$ is a subspace spanned by $$\{(a^1\otimes 1)_{n_1}(a^2\otimes 1)_{n_2}\cdots (a^{r}\otimes 1)_{n_r}({\bf 1}\otimes t)_{n}(a^{r+1}\otimes 1)_{n_{r+1}}\cdots (a^s\otimes 1)|a^i\in V,n,n_i\in\mathbb{Z}, r,s\in\mathbb{Z^+}\}.$$

Let $e=\sum t_{1}c\otimes t_{2}\in V^H$ be the element defined in Proposition \ref{prop37}. For any representation $M$ of $(V,Y,{\bf 1})$, let $N=(V\#H)\otimes M$. Then $N$ become a representation of $V\#H$ via
 $$Y^N(a\otimes h,z)(b\otimes g\otimes m)=Y(a\otimes h,z)(b\otimes g)\otimes m.$$
 For any subrepresentation $U$ of $M$, define $\sigma(U):=(V\#H)_{-1}e\otimes U.$ Conversely, for any  subrepresentation  $U'$ of $W$,  set $\tau (\sum\limits_ih_i\otimes v_i)=
 \sum\limits_{i}\varepsilon(h_i)v_i$.


\section{The smash product $V\#H$}
In this section, we investigate the construction of the smash product of an $H$-module vertex algebra $V$ and the Hopf algebra $H$.

The definition of $H$-module vertex algebra can be regarded as the generalization of the $H$-module vertex operator algebra in \cite{HW}. It is well-known that the smash product of an $H$-module algebra $V$ and a Hopf algebra $H$ is still an associative algebra in \cite{DNR}. We can obtain the following theorem and propositions.
\begin{theorem}\label{tm3.2}
Let $(V, Y, {\bf 1}, s)$ be an $H$-module vertex algebra and $H$ be a Hopf algebra. Then, the smash product $(V\#H, Y^{V\#H}, {\bf 1} \#1, s\#1)$ is an $\mathcal{S}$-local vertex algebra, where $Y^{V\#H}$ is defined as follows:
\begin{align}\label{eq3.1}
Y^{V\#H}(a\#h,z)(b\#g)=\sum Y(a,z)(h_1b)\#h_2g,
\end{align}
for any $a \# h, b\#g \in V\#H$, where $\Delta(h)=\sum h_1\otimes h_2$. 
\end{theorem}
\begin{proof}
By Lemma \ref{lm2.15}, we can obtain that $V\#H$ is a field algebra. We only need to check the $\mathcal{S}$-locality identity.

For any $a\#h, b\#g, c\#k \in V\#H$, we have
\begin{align*}
Y^{V\#H}(a\#h,z)Y^{V\#H}(b\#g,w)(c\#k)=\sum Y(a,z)Y(h_1b,w)(h_2g_1c\#h_3g_2k),
\end{align*}
and
\begin{align*}
Y^{V\#H}(h_1b\#1,w)Y^{V\#H}(a\#h_2g,z)(c\#k)=\sum Y(h_1b,w)Y(a,z)(h_2g_1c\#h_3g_2k).
\end{align*}
Since $\Delta(h)=\sum h_1\otimes h_2$ is a finite sum, there is $n \in  \mathbb{Z^+}$ such that 
\begin{align*}
(z-w)^nY(a,z)Y(h_1b,w)=(z-w)^nY(h_1b,w)Y(a,z),
\end{align*}
for all $h_1$. 

This completes the proof.
\end{proof}

\begin{remark}
Above theorem is similar to Theorem 3.3 in \cite{HW} where $V$ is an $H$-module vertex operator algebra.
\end{remark}

Together with Lemma \ref{lm5.1} and Theorem \ref{tm3.2} we can obtain the following result immediately.
\begin{corollary}\label{cor4.2}
Let $V$ be an $H$-module vertex algebra and $H={\bf k}G$, where $G$ is a finite subgroup of $Aut(V)$. Then $(V\#H\#H^*, Y^{V\#H\#H^*}, {\bf 1} \# 1\# E, s\#1\#E)$ is an $\mathcal{S}$-local vertex algebra, where $Y^{V\#H\#H^*}$ is defined as follows:
\begin{align}\label{eq5.2}
Y^{V\#H\#H^*}(u \# g\# \rho_a,z)(v\#h\#\rho_b)=\delta_{a,hb}Y(u,z)gv\#gh\#\rho_b,
\end{align}
for any $u \# g\# \rho_a, v\#h\#\rho_b \in V\#H\#H^*$ and extend to $V\#H\#H^*$ linearly. 
\end{corollary} 

The construction of the smash product of an $H$-module vertex operator algebra $V$ and an $H$-module $M$ was investigated by Wang in \cite{HW}. He also characterized the connection between the smash product of an $H$-module vertex operator algebra $V$ as well as an $H$-module $M$ and a $V\#H$-module. We prove the similar result with respect to the $H$-module vertex algebra $V$.

\begin{proposition}\label{pro3.4}
If $M$ is an $H$-module and $V$ is an $H$-module vertex algebra, then $(V\otimes M, Y_{V\otimes M}, s_{V\otimes M})$ is a $V\#H$-module with 
\begin{align*}
Y_{V\otimes M}(a\#h,z)(b \otimes m)=\sum Y(a,z)h_1b \otimes h_2m,\qquad s_{V\otimes M}=s \otimes Id,
\end{align*}
for any $a,b \in V$, $h \in H$, $m \in M$.

Conversely, if $(M, Y_M, s_ M)$ is a $V\#H$-module, then it is an $H$-module and a $V$-module with 
\begin{align*}
hm=Res_zz^{-1}Y_M({\bf 1} \#h,z)m, \qquad Y_M^V(a,z)m=Y_M(a\#1,z)m,\qquad s_M^V=s_M,
\end{align*}
for any $a \in V$, $h \in H$, $m \in M$.
\end{proposition}
\begin{proof}
If $M$ is an $H$-module and $V$ is an $H$-module vertex algebra, we need to check all the axioms in the Definition \ref{def2.9}.

(i)For any $a \#h \in V\#H$, $b \otimes m \in V \otimes M$, $n \in \mathbb{Z}$, $(a\#h)_n(b \otimes m)=a_n(h_1b) \otimes h_2m=0$ if $n\gg 0$. And $Y_{V\otimes M}({\bf 1}\#1,z)(b \otimes m)=Y({\bf 1},z)b\otimes m=b\otimes m.$

(ii)For any $a \#h \in V\#H$, $b \otimes m \in V \otimes M$, we can obtain that 
\begin{align*}
&s_{V\otimes M}Y_{V\otimes M}(a\#h,z)(b\otimes m)-Y_{V\otimes M}(a\#h,z)(s_{V\otimes M}(b\otimes m))\\
&=(s\otimes Id)\sum Y(a,z)(h_1b)\otimes h_2m-Y_{V\otimes M}(a\#h,z)(sb\otimes m)\\
&=\sum sY(a,z)(h_1b)\otimes h_2m-\sum Y(a,z)(h_1sb)\otimes h_2m\\
&=\sum \partial_zY(a,z)(h_1b)\otimes h_2m\\
&=\partial_zY_{V\otimes M}(a\#h,z)(b\otimes m).
\end{align*}

(iii) On the one hand, for any $a \#h, b\#g \in V\#H$, $c \otimes m \in V \otimes M$, we can obtain that 
\begin{align*}
&Y_{V\otimes M}(a\#h,z-w)Y_{V\otimes M}(b\#g,-w)(c\otimes m)\\
&=Y_{V\otimes M}(a\#h,z-w)\sum Y(b,-w)(g_1c)\otimes g_2m\\
&=\sum Y(a,z-w)h_1 Y(b,-w)(g_1c)\otimes h_2g_2m\\
&=\sum Y(a,z-w)Y(h_1b,-w)(h_2g_1c)\otimes h_3g_2m.
\end{align*}

On the other hand, we can obtain that
\begin{align*}
&Y_{V\otimes M}(Y^{V\#H}(a\#h, z)(b\#g),-w)(c\otimes m)\\
&=Y_{V\otimes M}(\sum Y(a,z)(h_1b)\# h_2g,-w)(c\otimes m)\\
&=\sum Y(Y(a,z)(h_1b),-w)(h_2g_1c)\otimes h_3g_2m.
\end{align*}

Since $V$ is a vertex algebra and $\Delta(h)=\sum h_1\otimes h_2$ is a finite sum, there is $n \in  \mathbb{Z^+}$ such that 
\begin{align*}
(z-w)^nY(a,z-w)Y(h_1b,-w)=(z-w)^nY(Y(a,z)(h_1b),-w),
\end{align*}

for all $h_1$.
Thus, we obtain that
\begin{align*}
&(z-w)^n\sum Y(a,z-w)Y(h_1b,-w)(h_2g_1c)\otimes h_3g_2m\\
&=(z-w)^n\sum Y(Y(a,z)(h_1b),-w)(h_2g_1c)\otimes h_3g_2m.
\end{align*}

(iv) For any $a \#h, b\#g \in V\#H$, $c \otimes m \in V \otimes M$, we can obtain that 
\begin{align*}
&Y_{V\otimes M}(a\#h, z)Y_{V\otimes M}(b\#g,w)(c\otimes m)\\
&=\sum Y(a,z)Y(h_1b,w)(h_2g_1c)\otimes h_3g_2m.
\end{align*}
and
\begin{align*}
&\sum Y_{V\otimes M}(h_1b\#1, w)Y_{V\otimes M}(a\#h_2g,z)(c\otimes m)\\
&=\sum Y(h_1b,w)Y(a,z)(h_2g_1c)\otimes h_3g_2m.
\end{align*}
Since $V$ is a vertex algebra and $\Delta(h)=\sum h_1\otimes h_2$ is a finite sum, there is $n \in  \mathbb{Z^+}$ such that 
\begin{align*}
(z-w)^nY(a,z)Y(h_1b,w)=(z-w)^nY(h_1b,w)Y(a,z).
\end{align*}
for all $h_1$. Hence, we can obtain that
\begin{align*}
(z-w)^nY_{V\otimes M}(a\#h, z)Y_{V\otimes M}(b\#g,w)=(z-w)^n\sum Y_{V\otimes M}(h_1b\#1, w)Y_{V\otimes M}(a\#h_2g,z).
\end{align*}
Thus, $V\otimes M$ is a $V\# H$-module as $\mathcal{S}$-local vertex algebra.
Conversely, first, we need to check that $M$ is a $V$-module. It is easy to check that $a_nm=0$ for any $a \in V, m \in M$ if $n \gg 0$
and $Y_M^V({\bf 1},z)m=Y_M({\bf 1} \#1,z)m=m.$ For translation invariance, we can get 

\begin{align*}
&s_M^VY_M^V(a,z)m-Y_M^V(a,z)(s_M^V(m))\\
&=s_M^VY_M(a\#1,z)m-Y_M(a\#1,z)(s_M^V(m))\\
&=s_MY_M(a\#1,z)m-Y_M(a\#1,z)(s_M(m))\\
&=\partial_zY_M(a\#1,z)m\\
&=\partial_zY_M^V(a,z)m,
\end{align*}
for any $a \in V, m \in M$. Next, we should prove that $M$ is an $H$-module. For any $h,g \in H, m \in M$, we obtain that
\begin{align*}
&h(gm)-(hg)m\\
=&Res_{z,w}z^{-1}w^{-1}(Y_M({\bf 1} \# h,z)Y_M({\bf 1} \# g,w)-Y_M({\bf 1} \# 1,w)Y_M({\bf 1} \# hg,z))m\\
=&Res_{x,z,w}z^{-1}w^{-1}x^{-1}\delta(\frac{z-w}{x})Y_M({\bf 1} \# h,z)Y_M({\bf 1} \# g,w)m\\
&-Res_{x,z,w}z^{-1}w^{-1}x^{-1}\delta(\frac{-w+z}{x})Y_M({\bf 1} \# 1,w)Y_M({\bf 1} \# hg,z)m\\
=&Res_{x,z,w}z^{-1}w^{-1}z^{-1}\delta(\frac{w+x}{z})Y_M(Y^{V\#H}({\bf 1} \# h,x){\bf 1} \# g,w)m\\
=&Res_{x,z,w}z^{-1}w^{-1}z^{-1}\delta(\frac{w+x}{z})Y_M({\bf 1} \# hg,w)m=0.
\end{align*}
Hence, $h(gm)=(hg)m$.

This completes the proof.
\end{proof}

As in \cite{HW}, Wang investigated the connection of a $V$-module as well as an $H$-module and a $V\#H$-module when $V$ is a $H$-module vertex operator algebra. With the similar method, we can obtain the following two lemmas.
\begin{lemma}\label{lm4.1}
Let $(V, Y, {\bf 1}, s)$ be an $H$-module vertex algebra. If $(M, Y_M, s_M)$ is a $V\#H$-module, then for any $a \in V, h \in H, m \in M$, we have
\begin{enumerate}[(i)]
\item $hY_M(a,z)m=\sum Y_M(h_1a,z)h_2m$,
\item $Y_M(a\#h,z)m=Y_M(a,z)hm$.
\end{enumerate}
Here we use $Y_M$ to denote both $V$-module and $V\#H$-module.
\end{lemma}
\begin{proof}
(i) For any $a \in V, h \in H, m \in M$, we have
\begin{align*}
&hY_M(a,w)m-\sum Y_M(h_1a,w)h_2m\\
=&Res_{z}z^{-1}(Y_M({\bf 1} \#h,z)Y_M(a\#1,w)-\sum Y_M(h_1a\#1,w)Y_M({\bf 1}\#h_2,z))m\\
=&Res_{x,z}z^{-1}x^{-1}\delta(\frac{z-w}{x})Y_M({\bf 1} \#h,z)Y_M(a\#1,w)m\\
&-Res_{x,z}z^{-1}x^{-1}\delta(\frac{-w+z}{x})\sum Y_M(h_1a\#1,w)Y_M({\bf 1}\#h_2,z)m\\
=&Res_{x,z}z^{-1}z^{-1}\delta(\frac{w+x}{z})Y_M(Y^{V\#H}({\bf 1} \#h,x)a\#1,w)m\\
=&Res_{x,z}z^{-1}z^{-1}\delta(\frac{w+x}{z})Y_M(\sum h_1a\#h_2,w)m=0.
\end{align*}
(ii) For any $v \in V, h \in H, m \in M$, we have
\begin{align*}
&Y_M(v,z)hm-Y_M(v\#h,z)m\\
=&Res_{w}w^{-1}(Y_M(v\#1,z)Y_M({\bf 1} \#h,w)-Y_M({\bf 1}\#1,w)Y_M(v\#h,z))m\\
=&Res_{x,w}w^{-1}x^{-1}\delta(\frac{z-w}{x})Y_M(v\#1,z)Y_M({\bf 1} \#h,x)m\\
&-Res_{x,w}w^{-1}x^{-1}\delta(\frac{-w+z}{x})Y_M({\bf 1}\#1,w)Y_M(v\#h,z)m\\
=&Res_{x,w}w^{-1}z^{-1}\delta(\frac{w+x}{z})Y_M(Y^{V\#H}(v \#1,x){\bf 1}\#h,w)m\\
=&Res_{x,w}w^{-1}z^{-1}\delta(\frac{w+x}{z})Y_M(Y(v,x){\bf 1}\#h,w)m=0.
\end{align*}
This completes the proof.
\end{proof}

\begin{lemma}\label{lm4.2}
Let $(V, Y, {\bf 1}, s)$ be an $H$-module vertex algebra and $(M, Y_M, s_M)$ be a $V$-module as well as an $H$-module such that 
\begin{align*}
h(Y_M(a,z)m)=\sum Y_M(h_1a,z)h_2m, \quad hs_M(m)=s_M(hm), 
\end{align*}
for any $h\in H$, $a\in V$, $m \in M$, where $\Delta(h)=\sum h_1\otimes h_2$. Then $(M, \mathcal{Y},s_M)$ is a $V\#H$-module, where $\mathcal{Y}$ is defined as follows:
\begin{align*}
\mathcal{Y}(a\#h,z)m=Y_M(a,z)hm
\end{align*}
 for any $h\in H$, $a\in V$, $m \in M$.
\end{lemma}
\begin{proof}
We just prove the $\mathcal{S}$-locality identity, and other axioms follow from Lemma \ref{lm2.16}. For any $h,g\in H$, $a,b\in V$, $m \in M$, we have
\begin{align*}
\mathcal{Y}(a\#h,z)\mathcal{Y}(b\#g,w)m=\sum Y_M(a,z)Y_M(h_1b,w)h_2gm,\\
\sum \mathcal{Y}(h_1b\#1,w)\mathcal{Y}(a\#h_2g,z)m=\sum Y_M(h_1b,w)Y_M(a,z)h_2gm.
\end{align*}
Since $M$ is a $V$-module and $\Delta(h)=\sum h_1\otimes h_2$ is a finite sum, there is $n \in  \mathbb{Z^+}$ such that 
\begin{align*}
(z-w)^nY_M(a,z)Y_M(h_1b,w)=(z-w)^nY_M(h_1b,w)Y_M(a,z).
\end{align*}
for all $h_1$. Thus, we obtain that 
\begin{align*}
(z-w)^n\mathcal{Y}(a\#h,z)\mathcal{Y}(b\#g,w)m=(z-w)^n\sum \mathcal{Y}(h_1b\#1,w)\mathcal{Y}(a\#h_2g,z)m.
\end{align*}
The $\mathcal{S}$-locality identity holds.

This completes the proof.
 \end{proof}

 By Theorem \ref{tm2.6} and \ref{tm3.2}, there is an associative algebra $A(V\#H)$. On the other hand, it is easy to check that $H$ preserves $V_{(g)}V$ and $A(V)$ is an $H$-module. Thus, $A(V)$ is an $H$-module algebra. So $A(V)\#H$ is also an associative algebra. Then, we obtain the following result.

\begin{theorem}\label{tm3.5}
$A(V\#H)\cong A(V)\#H$ as associative algebras.
\end{theorem}
\begin{proof}
Define
\begin{align*}
\varphi : A(V)\#H \to A(V\#H),\quad [v]\#h \mapsto [v\#h],\\
\phi : A(V\#H) \to A(V)\#H ,\quad [v\#h]\mapsto [v]\#h.
\end{align*}
It is obvious that $\varphi \circ \phi=Id$, $\phi \circ \varphi=Id$.

Now, we only need to check that  $\varphi, \phi$ are well-defined algebra homomorphisms. Let $u,v \in V $, $h,g \in H$, we can obtain that 
\begin{align*}
&\varphi([u]\#h)_{(f)}\varphi([v]\#g)\\
&=[u\#h]_{(f)}[v\#g]\\
&=[Res_zf(z)Y^{V\#H}(u\#h,z)v\#g]\\
&=[Res_zf(z)\sum Y(u,z)(h_1v)\#h_2g]\\
&=\sum [u_{(f)}(h_1v)\#h_2g].
\end{align*}
On the other hand, we obtain that
\begin{align*}
&\varphi([u]\#h_{(f)}[v]\#g)\\
&=\varphi(Res_zf(z)Y^{A(V)\#H}([u]\#h,z)[v]\#g)\\
&=\varphi(Res_zf(z)\sum Y^{A(V)}([u],z)[h_1v]\#h_2g)\\
&=[Res_zf(z)\sum Y(u,z)(h_1v)\#h_2g]\\
&=\sum [u_{(f)}(h_1v)\#h_2g].
\end{align*}
Hence, $\varphi([u]\#h)_{(f)}\varphi([v]\#g)=\varphi([u]\#h_{(f)}[v]\#g)$. 

Similarly, we can prove that $\phi([u\#h])_{(f)}\phi([v\#g])=\phi([u\#h]_{(f)}[v\#g])$.

This completes the proof.
\end{proof}

One can deduce the following result by Example \ref{ex2.8}, Theorem \ref{tm2.6}, Corollary \ref{cor4.2} and Theorem \ref{tm3.5} immediately.

\begin{corollary}\label{cor4.7}
For each $n \in \mathbb{Z^+}$, we have:
\begin{enumerate}[(i)]
\item $A(V\#H\#H^*)\cong A(V)\#H\#H^*$;
\item $A(M(n, V))\cong A(V)\otimes M(n, {\bf k})\cong M(n, A(V))$.
\end{enumerate}
 \end{corollary}
 

\section{Quantization of $\mathcal{S}$-local vertex algebra}
Throughout the rest of the paper we shall work over the algebra ${\bf k}[[h]]$ of formal power series in the variable $h$, and all the algebraic structures that we will consider are modules over ${\bf k}[[h]]$. 

First, we recall some basic results of \emph{topologically free} ${\bf k}[[h]]$-module for later use.
\begin{definition}\label{def6.1}(Ref. \cite{DGK})
\begin{em}
A \emph{topologically free} ${\bf k}[[h]]$-module is isomorphic to $W[[h]]$ for some vector space $W$ over ${\bf k}$.
\end{em}		
\end{definition}
\begin{remark}\label{rm6.2}(Ref. \cite{DGK})
Note that $W[[h]]\ncong W \otimes {\bf k}[[h]]$, unless $W$ is finite-dimensional vector space over ${\bf k}$, and that the tensor product $U[[h]] \otimes_{{\bf k}[[h]]} W[[h]]$ of topologically free ${\bf k}[[h]]$-modules is not topologically free, unless one of $U$ and $W$ is finite dimensional. One defines, for any vector spaces $U$ and $W$, the \emph{completed} tensor product by 
\begin{align}\label{eq6.1}
U[[h]] \widehat{\otimes}_{{\bf k}[[h]]} W[[h]] :=(U \otimes W)[[h]].
\end{align}
This is a completion in $h$-adic topology of $U[[h]] \otimes_{{\bf k}[[h]]} W[[h]]$ (see \cite{Kas}).
\end{remark}

Let $V$ be a topologically free ${\bf k}[[h]]$-module. An \emph{$End_{{\bf k}[[h]]}V$-valued quantum field} is an $End_{{\bf k}[[h]]}V$-valued formal distribution $Y(a,z)$ such that $Y(a,z)b \in V_h((z))$ for any $b \in V$, where 
\begin{align}\label{eq6.2}
V_h((z))=\{Y(a,z) \in V[[z,z^{-1}]] | Y(a,z) \in V[[z]][z^{-1}]\  mod\  h^M for\  every\  M \in \mathbb{Z}_{\ge 0}\}.
\end{align}

Now, we propose the definition of braided vertex algebra in \cite{DGK}.
\begin{definition}\label{def6.4}(Ref. \cite{DGK})
\begin{em}
Let $V$ be a topologically free ${\bf k}[[h]]$-module, with a given non-zero vector ${\bf 1} \in V$ (vacuum vector), and a ${\bf k}[[h]]$-linear map $s: V \to V$ such that $s({\bf 1})=0$ (translation operator).
\begin{enumerate}[(a)]
\item A \emph{topological state-field} correspondence on $V$ is a ${\bf k}[[h]]$-linear map
\begin{align}\label{eq6.3}
Y: V \widehat{\otimes} V \to V_h((z)),
\end{align}
satisfying the axioms of a state-field correspondence as in Definition \ref{def2.31}.

\item A \emph{braiding} on $V$ is a ${\bf k}[[h]]$-linear map
\begin{align}\label{eq6.4}
\mathcal{Q}: V \widehat{\otimes} V \to V \widehat{\otimes} V  \widehat{\otimes}  ({\bf k}((z))[[h]]),
\end{align}
sucth that $\mathcal{Q}=1+O(h)$.
\end{enumerate}
\end{em}		
\end{definition}

\begin{definition}\label{def6.5}(Ref. \cite{DGK})
\begin{em}
A \emph{braided vertex algebra} is a quintuple $(V, {\bf 1}, s, Y, \mathcal{Q})$ as in Definition \ref{def6.4}, satisfying the following $\mathcal{Q}$-locality: for any $a,b \in V$ and $M \in \mathbb{Z}_{\ge 0} $, there exists $n=n(a,b,M) \in  \mathbb{Z^+}$ such that 
\begin{eqnarray}\label{eq6.5}
(z-w)^nY(a,z)Y(b,w)c \cdot \mathcal{Q}(z-w)(a \otimes b)=(z-w)^nY(b,w)Y(a,z)c\ mod \ h^M
\end{eqnarray}
for all $c \in V$.
\end{em}		
\end{definition}

Next, we give the definition of quantum vertex algebra in \cite{DGK} which is similar to the definition of vertex algebra.
\begin{definition}\label{def6.6}(Ref. \cite{DGK})
\begin{em}
A \emph{quantum vertex algebra}  is a braided vertex algebra satisfying the \emph{weak associativity relation}: for any $a,b,c\in V$ and $M \in \mathbb{Z}_{\ge 0} $, there exists $n \in  \mathbb{Z}_{\ge 0}$ such that 
\begin{eqnarray}\label{eq6.6}
(z-w)^nY(a,z-w)Y(b,-w)c=(z-w)^nY(Y(a,z)b,-w)c  \ mod \ h^M.
\end{eqnarray}
\end{em}		
\end{definition}

Finally, we extend the $\mathcal{S}$-locality to the quantum vertex algebras. We can obtain the definition of $\mathcal{S}$-local quantum vertex algebra.
\begin{definition}\label{def6.7}
\begin{em}
An \emph{$\mathcal{S}$-local quantum vertex algebra} is a quintuple $(V, {\bf 1}, s, Y, \mathcal{Q})$ as in Definition \ref{def6.4}, satisfying the \emph{weak associativity relation (\ref{eq6.6})} and the following $\mathcal{QS}$-locality: for any $a,b \in V$ and $M \in \mathbb{Z}_{\ge 0} $, there exists $n=n(a,b,M), m \in  \mathbb{Z^+}$ and $a^i,b^i \in V, i=1,2,...,m$ such that 
\begin{align}\label{eq6.7}
(z-w)^nY(a,z)Y(b,w)c \cdot \mathcal{Q}(z-w)(a \otimes b)=(z-w)^n\sum_{i=1}^mY(b^i ,w)Y(a^i,z)c\ mod \ h^M
\end{align}
for all $c \in V$.\end{em}		
\end{definition}

Let $(V, Y,{\bf 1}, s, \mathcal{Q})$ be a quantum vertex algebra and $H$ a Hopf algebra. Then $V$ is called \emph{$H$-module quantum vertex algebra} if satisfies all axioms as in Definition \ref{def3.1}. Similar to Theorem \ref{tm3.2}, we can obtain the following result.
\begin{theorem}\label{tm6.9}
Let $(V, Y, {\bf 1}), s, \mathcal{Q})$ be an $H$-module quantum vertex algebra. Define $Y^{V\#H}(\cdot,z)$ as follows:
\begin{align}\label{eq6.8}
Y^{V\#H}(u\#a,z)(v\#b)=\sum Y(u,z)(a_1v)\#a_2b,
\end{align}
where $\Delta(a)=\sum a_1\otimes a_2$. 

Then, the smash product $(V\#H, Y^{V\#H}, {\bf 1} \#1, s\#1, \mathcal{Q}^{V\#H})$ is an $\mathcal{S}$-local quantum vertex algebra if $\mathcal{Q}^{V\#H}$ meets the following conditions: for any $(u\#a)\otimes (v\#b) \in (V\#H) \widehat{\otimes} (V\#H)$, 
\begin{align}\label{eq6.9}
&\mathcal{Q}^{V\#H}((u\#a)\otimes (v\#b))=(u\#a)\otimes (v\#b)\otimes F_h(u\#a,v\#b),\\
&\mathcal{Q}(u\otimes a_1v)=u\otimes a_1v\otimes G_h(u,a_1v),
\end{align}
where $F_h(u\#a,v\#b), G_h(u,a_1v)\in {\bf k}((z))[[h]]$, then $F_h(u\#a,v\#b)=G_h(u,a_1v)$ for all $a_1$, where $\Delta(a)=\sum a_1\otimes a_2$.
\end{theorem}

\begin{proof}
By Theorem \ref{tm3.2}, we only need to check the associativity relation (\ref{eq6.6}) and $\mathcal{QS}$-locality identity (\ref{eq6.7}).

For any $u\#a, v\#b, w\#c \in V\#H$, we have
\begin{align*}
&Y^{V\#H}(u\#a,z_0)Y^{V\#H}(v\#b,z_1)w\#c \cdot \mathcal{Q}^{V\#H}(z_0-z_1)(u\#a \otimes v\#b)\\
&=Y^{V\#H}(u\#a,z_0)Y^{V\#H}(v\#b,z_1)w\#c \cdot F_h(u\#a,v\#b)\\
&=\sum Y(u,z_0)Y(a_1v, z_1)(a_2b_1w)\# a_3b_2c \cdot F_h(u\#a,v\#b)\ mod \ h^M,
\end{align*}
and
\begin{align*}
&Y^{V\#H}(a_1v\#1 ,z_1)Y^{V\#H}(u\#a_2b,z_0) w\#c\\
&=\sum Y(a_1v, z_1)Y(u,z_0)(a_2b_1w)\# a_3b_2c\ mod \ h^M,
\end{align*}
Since $\Delta(a)=\sum a_1\otimes a_2$ is a finite sum and $V$ is an $H$-module quantum vertex algebra, there is $n \in  \mathbb{Z^+}$ such that 
\begin{align*}
(z_0-z_1)^nY(u,z_0)Y(a_1v,z_1) \cdot G_h(u,a_1v)=(z_0-z_1)^nY(a_1v,z_1)Y(u,z_0)\ mod \ h^M,
\end{align*}
for all $a_1$. Together with the assumption of $\mathcal{Q}^{V\#H}$, i.e., $F_h(u\#a,v\#b)=G_h(u,a_1v)$ for all $a_1$, we can deduce that 
\begin{align*}
&(z_0-z_1)^nY^{V\#H}(u\#a,z_0)Y^{V\#H}(v\#b,z_1)w\#c \cdot \mathcal{Q}^{V\#H}(z_0-z_1)(u\#a \otimes v\#b)\\
&=(z_0-z_1)^n\sum Y^{V\#H}(a_1v\#1 ,z_1)Y^{V\#H}(u\#a_2b,z_0) w\#c \ mod \ h^M.
\end{align*}
The $\mathcal{QS}$-locality identity follows.

For any $u\#a, v\#b, w\#c \in V\#H$, we have
\begin{align*}
&Y^{V\#H}(u\#a,z_0+z_1)Y^{V\#H}(v\#b,z_1) w\#c\\
&=Y^{V\#H}(u\#a,z_0+z_1)\sum Y(v,z_1) b_1w\#b_2c\\
&=\sum Y(u,z_0+z_1)a_1 \sum Y(v,z_1) b_1w\#a_2b_2c\\
&=\sum Y(u,z_0+z_1)Y(a_1 v,z_1) a_2b_1w\#a_3b_2c,
\end{align*}
and 
\begin{align*}
&Y^{V\#H}(Y^{V\#H}(u\#a,z_0)v\#b,z_1) w\#c\\
&=Y^{V\#H}(\sum Y(u,z_0)a_1v\#a_2b,z_1) w\#c\\
&=\sum Y(Y(u,z_0)a_1v,z_1) a_2b_1w\#a_3b_2c.
\end{align*}
Since $\Delta(a)=\sum a_1\otimes a_2$ is a finite sum and $V$ is an $H$-module quantum vertex algebra, there is $n \in  \mathbb{Z^+}$ such that
\begin{align*}
(z_0+z_1)^nY(u,z_0+z_1)Y(a_1v,z_1)a_2b_1w=(z_0+z_1)^nY(Y(u,z_0)a_1v,z_1)a_2b_1w \ mod \ h^M.
\end{align*}
Hence, we can obtain that 
\begin{align*}
&(z_0+z_1)^nY^{V\#H}(u\#a,z_0+z_1)Y^{V\#H}(v\#b,z_1) w\#c\\
&=(z_0+z_1)^nY^{V\#H}(Y^{V\#H}(u\#a,z_0)v\#b,z_1) w\#c \ mod \ h^M.
\end{align*}
Thus, the associativity relation holds.

This completes the proof.
\end{proof}


%
\end{document}